\documentclass[absolute]{mymathart}
\usepackage[theorems]{mymathmacros}
\usepackage[usenames,dvipsnames]{color}
\usepackage{subfig}
\usepackage[shortlabels]{enumitem}
\usepackage{hyperref}
\usepackage{amsmath}
\usepackage{amssymb}
\usepackage{graphicx}
\usepackage{epigraph}
\usepackage{amsfonts,amssymb,color}
\usepackage[mathscr]{eucal}
\usepackage{amsmath, amsthm}
\usepackage{mathrsfs}
\usepackage{amsbsy}
\usepackage{float}
\usepackage{verbatim}
\usepackage{amsfonts} 
\usepackage{pifont}

\title[Counting the Lyapunov inflections]{Counting the Lyapunov inflections in piecewise linear systems}

\author{Liangang Ma}

\address{Dept.\ of Mathematical Sciences, Binzhou University, Huanghe 5th Road No. 391, Binzhou 256600, Shandong, P. R. China} 
\email{maliangang000@163.com}
\thanks{The work is supported by ZR2019QA003 from SPNSF and 2016Y28 from Binzhou University.}  
\newtheorem{theorem}[subsection]{Theorem}
\newtheorem{lemma}[subsection]{Lemma}
\newtheorem{pro}[subsection]{Proposition}
\newtheorem{coro}[subsection]{Corollary}
\newtheorem{exa}[subsection]{Example}
\newtheorem{rem}[subsection]{Remark}

\numberwithin{equation}{section}

\begin{document} 

\begin{abstract}
   
   Following the pioneering work of Iommi-Kiwi and Jenkinson-Pollicott-Vytnova, we continue to study the inflection points of the Lyapunov spectrum in this work. We prove that for any $3$-branch piecewise linear expanding map on an interval, the number of its Lyapunov inflections is bounded above by $2$. Then we continue to show that, there is a $4$-branch piecewise linear expanding map, such that its Lyapunov spectrum has exactly  $4$ inflection points. These results give an answer to a question by Jenkinson-Pollicott-Vytnova on the least number of branches needed to observe $4$ inflections in the Lyapunov spectrum of piecewise linear maps. In the general case, we give upper bound on the number of Lyapunov inflections for any $n$-branch piecewise linear expanding maps, and  construct a family of $n$-branch piecewise linear expanding maps with $2n-4$ Lyapunov inflections. We also consider the number of Lyapunov inflections of piecewise linear maps in words of the essential branch number in this work. There are some results on distributions of the Lyapunov inflections of piecewise linear maps through out the work, in case of their existence.

 \end{abstract}
 
 \maketitle

\section{Introduction}\label{sec1}

This work fits into multifractal analysis for Birkhoff averages, which involves describing sizes of level sets of some set by their Hausdorff dimensions. The notion of Lyapunov spectrum for a differentiable dynamical system  was introduced by J. P. Eckmann and I. Procaccia \cite{EP}.  The basic idea is to slice the domain of the map according to the asymptotic averages of the logarithmic modules of derivatives of iterations of the map, and then to decide the Hausdorff dimensions of the slices. Its rigorous study dates back to H. Weiss \cite{Wei}. Upon the techniques developed by Y. Pesin \cite{Pes} and Pesin-Weiss \cite{PesW}, he showed that for a conformal expanding (hyperbolic) map with a compact repeller, its  Lyapunov spectrum is smooth and has bounded domain. This regularity result is quite surprising considering the intricate interlacements of points in the level sets. Since then there are various study on the Lyapunov spectrum in more violent circumstances. For example, for non-hyperbolic maps or maps with non-compact repellers, see \cite{FLWW, GR, Iom, KS, Nak, PolW, PS, TV}. We will not deal with the  Lyapunov spectrum in these more sophisticated circumstances, but just point out that some pathological phenomenons may happen. For example, not only smoothness of the spectrum may be terribly disrupted, but also the boundedness of its domain may fail, according to the above mentioned work.

Weiss claimed  the  Lyapunov spectrum  is always concave for conformal expanding maps with compact repellers. This is confirmed to be wrong, by a theorem of G. Iommi and J. Kiwi  \cite[Theorem A]{IK} (there are preceding examples with non-compact repellers, see for example \cite{Iom, KS}). Thus the second derivative of the  Lyapunov spectrum is possible to vanish at some points, which are called Lyapunov inflections of the spectrum, or directly of the dynamical system. We want to mention that inflection points of functions as well as the concavity and convexity behind them are also the focus of many economists. There are important applications of them into economical and financial activities, see for example \cite{HS}.   According to  \cite{IK}, the number of Lyapunov inflections for piecewise linear expanding maps with finitely many branches is even. As the inflection points  provide deeper sight into the geometric and topological behaviour of the Lyapunov spectrum, Iommi and Kiwi posed various questions on the possible number of the Lyapunov inflections. They conjecture that the spectrum has at most two Lyapunov inflections for conformal expanding maps with two branches. The answer is yes for piecewise linear maps, according to \cite[Theorem 1.3]{JPV}, but remains open in the case of nonlinear maps. Another question is on the bound of number of  the  Lyapunov inflections for the whole family of finitely branched expanding maps. O. Jenkinson, M. Pollicott and P. Vytnova have given examples to show that, even in the  piecewise linear case, the number of Lyapunov inflections is unbounded, see \cite[Theorem 1.4]{JPV}. To witness a small even number of Lyapunov inflections, the examples of maps in  \cite{JPV} cost relatively large numbers of branches. Then they ask about the most economical number of branches in order to observe a given even number of  Lyapunov inflections for piecewise linear maps, see \cite[Question 6.5]{JPV}. We will decide the most economical number of branches to observe $4$ Lyapunov inflections  in this work.

%The above question of Jenkinson-Pollicott-Vytnova drives our motivation for the work, but we also obtain some results of special interest during solutions of their problem. For example, 

By eliminating branches of the same slopes, we define the essential branch number of a piecewise linear expanding map, and then discuss the most possible number of Lyapunov inflections for piecewise linear maps with certain essential branch number. There are also some results on distribution of the  Lyapunov inflections for piecewise linear maps, on the domain of the spectrum or on the parametrized space. We mainly pay attention to piecewise linear maps in this work, although some results or notions may be applied to the nonlinear ones, either directly or after small modifications. The organization of the paper is as following. In Section \ref{sec2} we give some basic notations, definitions, present our main result-Theorem \ref{thm5}, \ref{thm1}, \ref{thm2} and \ref{thm7}.  In section \ref{sec3} we deduce the characteristic function which reflects the concavity and convexity of the Lyapunov spectrum for piecewise linear maps.    In section \ref{sec4} we prove Theorem \ref{thm5} and \ref{thm1}, which establish upper bounds on the number of Lyapunov inflections for piecewise linear expanding maps. In Section  \ref{sec5} we prove Theorem \ref{thm2} and \ref{thm7}, by the technique of root-surgery on some  piecewise linear maps. In section \ref{sec6} we get some results on locating positions of Lyapunov inflections, for example, relative to the  Lyapunov milstones.   There are various examples of piecewise linear expanding maps together with the graphs of their Lyapunov spectrums and characteristic functions provided through out the work.

\section{Some basic notations, definitions and the main theorems}\label{sec2}

For a countable index set $I$, let 
\begin{center}
$X=\cup_{i\in I}X_i$ 
\end{center}
be the union of some closed intervals $X_i\subset [0,1]$ with pairwise disjoint interiors. A differentiable map 
\begin{center}
$T: X\rightarrow X$ 
\end{center}
is called a \emph{cookie-cutter} map with $\#I$ \emph{branches} if $T|_{X_i}$ is a surjection from $X_i$ to $[0,1]$ for any $i\in I$, see \cite{Bed, BR, Sul}. One can also understand the map from the point view of an \emph{iterated function system}, see for example \cite{FH, MU}.  If every branch map $T|_{X_i}$ is affine on $X_i$, we call $T$  a \emph{linear cookie-cutter} or \emph{piecewise linear} map. Let 
\begin{center}
$|X_i|=x_i^{-1}$ 
\end{center}
be the length of the closed interval $X_i$ for any $i\in I$, then the slopes of a linear cookie-cutter or piecewise linear map need to satisfy
\begin{center}
$|T'(x)|=x_i$
\end{center}
for any $x\in X_i$. $T$ is said to be \emph{expanding} if 
\begin{center}
$\inf_{ x\in X}\{|T'(x)|\}>1.$
\end{center}
We are particularly interested in the \emph{repeller} $\Lambda$ of the map $T$, which is defined as the set
\begin{center}
$\Lambda=\cap_{i=0}^\infty T^{-i}(X)$.
\end{center}
It is a compact set if $\#I$ is finite, and may be non-compact for infinite $\#I$.  The \emph{Lyapunov exponent} of the map $T$ at $x\in \Lambda$ is defined to be
\begin{center}
$\lambda(x)=\lim_{n\rightarrow\infty} \frac{1}{n} \log |(T^n)'(x)|$,
\end{center}
in case the above limit exists. For $\alpha\in[-\infty,\infty]$, define the $\alpha$-\emph{level set} of the repeller via the Lyapunov exponent to be 
\begin{center}
$J(\alpha)=\{x\in\Lambda: \lambda(x)=\alpha\}.$
\end{center}
The value describes the average exponential divergent rate of  infinitesimally close orbits under iterations of $T$. We are only interested in values of $\alpha$ with non-empty level set $J(\alpha)$. So the repeller $\Lambda$ is sliced into the following union of sets,
\begin{center}
$\Lambda=\{x\in\Lambda: \lim_{n\rightarrow\infty} \frac{1}{n} \log |(T^n)'(x)| \mbox{ does not exist }\}\cup_{\alpha\in[-\infty,\infty]} J(\alpha)$.
\end{center}
The \emph{Lyapunov spectrum} is defined as 
\begin{center}
$L(\alpha)=\dim_H J(\alpha)$
\end{center}
for values $\alpha$ with non-empty level sets $J(\alpha)$. It is a function encodes the Hausdorff dimension (see \cite{Fal} for the definition) of each level-$\alpha$ set. Recall that we only require these closed intervals $\{X_i\}_{i=1}^{\#I}$ to have disjoint interiors, possibly with  some intersections at some of their endpoints. If $T$ is finitely branched and is a bijection restricted on each $X_i$, then the endpoints and their pre-images are countable, so they are Hausdorff dimensionless, which have no effect on the Lyapunov spectrum in this case. 

There are lots of work on the domain of the Lyapunov spectrum $L(\alpha)$, on its values as well as on the regularity of the function with respect to $\alpha$. If $T$ is a finitely branched expanding map, then the domain of $L(\alpha)$ (those $\alpha$ with non-empty level set $J(\alpha)$) is a closed interval in $(0,\infty)$, the spectrum $L(\alpha)$ is also smooth, according to \cite{Wei}. Otherwise the domain may not be bounded and the spectrum may not be smooth, see for example \cite{Iom}. These work relies heavily on the popular concepts and techniques in thermodynamic formalism. Let $\mathcal{M}_T$ be the set of all $T$-invariant Borel probability measures on $\Lambda$. For a measure $\mu\in\mathcal{M}_T$, let $h(\mu)$ be its \emph{measure-theoretic entropy} with respect to $T$, see \cite{Wal} for the definition. For a continuous potential $\phi: \Lambda\rightarrow \mathbb{R}$, define its \emph{topological pressure} to be
\begin{center}
$P(\phi)=\lim_{n\rightarrow\infty} \frac{1}{n} \log \sum_{T^n(x)=x} e^{\sum_{j=0}^{n-1}\phi(T^i(x))}$,
\end{center}  
see \cite{PP, Wal, Rue}. It satisfies the well known variational formula
\begin{center}
$P(\phi)=\sup_{\mu\in\mathcal{M}_T}\{h(\mu)+\int\phi d\mu\}.$
\end{center}
A measure $\mu$ is called an \emph{equilibrium measure} for $\phi$ if it satisfies 
\begin{center}
$P(\phi)=h(\mu)+\int\phi d\mu$. 
\end{center}

From now on we assume 
\begin{center}
$T: X=\cup_{i=1}^n X_i\rightarrow [0,1]$ 
\end{center}
is a finitely branched piecewise linear expanding map, with $X_i\subset [0,1]$ for any $1\leq i\leq n$,  unless otherwise stated. By rearranging the indexes of the branches if necessary, we always assume the length of the intervals $|X_i|=x_i^{-1}$ is in an non-increasing order, which means that
\begin{equation}\label{eq4}
1< x_1\leq x_2\leq\cdots\leq x_n.
\end{equation}

Define the \emph{essential branch number} of the map $T$ to be the integer 
\begin{center}
$r_T=\#\{x_i\}_{i=1}^n$.
\end{center}
It is possible that $r_T<n$ if the lengths of  some of the intervals equal each other, otherwise $r_T=n$ if the lengths of  all the intervals are distinct. Obviously the domain of the Lyapunov spectrum $L(\alpha)$ under our assumptions is 

\begin{center}
$[\log x_1,\log x_n]$. 
\end{center}
In case of existence of Lyapunov inflections, we want to locate their general positions in the interval  $[\log x_1,\log x_n]$. It tends out that the \emph{Lyapunov milestones} are potential markings to describe the distribution of the inflections in the domain. Let 
\begin{center}
$
\begin{array}{ll}
I_{r_T}=& \{i: 1\leq i\leq n, i\ \text{is the least index number of branches with equal lengths} \vspace{2mm}\\
 &\text{of their intervals} \}.
\end{array}
$
\end{center}
Obviously we have $\#I_{r_T}=r_T$. The \emph{Lyapunov milestones} are defined to be the points
\begin{center}
$\{\log x_i\}_{i\in I_{r_T}}$
\end{center}
in the interval $[\log x_1,\log x_n]$. These milestones divide the the domain of the Lyapunov spectrum into $r_T-1$ closed subintervals, we will consider how  the Lyapunov inflections are distributed in these subintervals, in case of their existence.

Our first main result establishes the upper bound on the number of Lyapunov inflections for general  $n$-branch piecewise linear maps. 

\begin{theorem}\label{thm5}
 For an $n$-branch piecewise linear expanding map with $n\geq 3$, the  number of its Lyapunov inflections  is less than or equal to $\cfrac{n(n-1)(n+4)}{6}$.
\end{theorem}

Then in case of $n=3$ we sharpen the bound to $2$. It is the best bound according to known examples, for example, the $3$-branch piecewise linear map $T_*$ in Section \ref{sec6}.
\begin{theorem}\label{thm1}
For any $3$-branch piecewise linear expanding map $T: X=\cup_{i=1}^3 X_i\rightarrow [0,1]$, the number of its Lyapunov inflections is less than or equal to $2$.
\end{theorem}
\begin{rem}
Considering the results \cite[Theorem A]{IK} and  \cite[Theorem 1.3]{JPV}, this theorem weakens one's expectation in counting more number of Lyapunov inflections for piecewise linear maps with increasing branches, which is quite a surprise to us.
\end{rem}

Then it is a natural question to ask that whether there exists a $4$-branch piecewise linear map with $4$ Lyapunov inflections. Our next result gives a positive answer to this question.
\begin{theorem}\label{thm2}
There exists a $4$-branch piecewise linear map $T: X=\cup_{i=1}^4 X_i\rightarrow [0,1]$, such that its Lyapunov spectrum has exactly $4$ inflections.
\end{theorem}

For general $n$-branch piecewise linear maps, we have the following result.   
\begin{theorem}\label{thm7}
For any $n\geq 5$, there exists an $n$-branch piecewise linear map $T: X=\cup_{i=1}^n X_i\rightarrow [0,1]$, such that its Lyapunov spectrum has at least $2n-4$ inflections.
\end{theorem}

Combining Theorem \ref{thm1} and Theorem \ref{thm2} together, we get the most economical number of branches to observe $4$ Lyapunov  inflections for a piecewise linear expanding map. This gives the answer to \cite[Question 6.5]{JPV} in case of $n=4$. Considering Theorem \ref{thm2}, we can sharpen the general upper bound in \cite[Corollary 6.6]{JPV}, while providing a lower bound by Theorem \ref{thm5}. 

\begin{coro}
In symbols of \cite[Question 6.5]{JPV}, we have
\begin{center}
$Q_4=4$ and $n_*\leq Q_n\leq \cfrac{n+4}{2}$ 
\end{center}
for any positive even integer $n$, in which $n_*$ is the unique real solution of the equation
\begin{center}
$\cfrac{x(x-1)(x+4)}{6}=n$.
\end{center} .
\end{coro}

In fact there is obviously a parallel question dual to \cite[Question 6.5]{JPV}: for an $n$-branch piecewise linear expanding map, what is the largest possible number $P_n$ of Lyapunov inflections it can admits?  Note that 
\begin{center}
$Q_{P_n}\leq n\leq P_{Q_n}$. 
\end{center}
The obvious bound for $P_n$ are
\begin{center}
$2n-4\leq P_n\leq \cfrac{n(n-1)(n+4)}{6}$
\end{center}
according to our Theorem \ref{thm5} and \ref{thm7}. In virtue of all the known results now, we have
\begin{center}
$P_2=P_3=2$.
\end{center}

We  also get some other interesting results on counting the number of Lyapunov inflections as well as on locating their positions, either in words of the essential branch number $r_T$ or the branch number $n$, for example, Theorem \ref{thm3} and \ref{thm6}, through out the paper.  We mainly focus on piecewise linear maps with low branches, however, the methods developed should be able to be applied or adapted to piecewise linear maps with higher branches, to achieve some delicate conclusions.

\section{The concavity-convexity characteristic function for piecewise linear maps}\label{sec3}

In this section we deduce a function which reflects the concavity or convexity of the Lyapunov spectrum for $n$-branch piecewise linear expanding maps. We follow the notations in Section \ref{sec2}, especially we remind the readers on our assumption (\ref{eq4}). The expression is explicit in this case, it appears first in some form in the work \cite[(17)]{JPV}, but we will deduce it in a different way as in \cite{JPV}. First we would like to introduce some descriptions on the Lyapunov spectrum, not only restricted to the piecewise linear case. For $\alpha$ in the domain of the Lyapunov spectrum, since the potential $-t\log |T'|$ is continuous, the pressure $P(-t\log |T'|)$ is differentiable, then the Lyapunov spectrum can be expressed as\\
\begin{center}
$
L(\alpha)=\frac{1}{\alpha}\inf_{t\in\mathbb{R}}\{P(-t\log |T'|)+t\alpha\}=\frac{h(\mu_{\alpha})}{\alpha}
=\frac{1}{\alpha}\sup_{\int |T'| d\mu=\alpha}h(\mu),
$
\end{center}
in which $\mu_{\alpha}$ is the equilibrium measure.  See \cite{FLW, Iom, IJ1, IJ2, IK, JPV, KS,  Wei}. In the case of piecewise linear maps, let 
\begin{center}
$F(t)=\sum_{i=1}^n x_i^t$
\end{center}
for $t\in[-\infty,\infty]$. Compare the notation with \cite[(15)]{JPV}, note that the parameter $t$ used here is inverse of the one used there. Iommi and Kiwi \cite[Lemma 4.1]{IK} deduced the following explicit formulas,
\begin{equation}\label{eq1}
\alpha(t)=\cfrac{F'(t)}{F(t)},
\end{equation}
\begin{equation}\label{eq2}
L(\alpha(t))=\cfrac{F(t)\log F(t)}{F'(t)}-t,
\end{equation} 
with the variable $\alpha$ be parametrized by the parameter $t\in[-\infty,\infty]$. Since $\alpha'(t)>0$ for any $t\in[-\infty,\infty]$, unless all the branches have the same length of their intervals, the graph of $L(\alpha(t))$ reflects most properties of the the graph of  $L(\alpha)$. 

Now we pay attention to the inflections of $L(\alpha)$. Iommi \cite[Remark 8.1, Lemma 8.1]{Iom} gave the first condition which guarantees the existence of a  Lyapunov inflection for \emph{MR-maps} as following, 
\begin{center}
$
\begin{array}{lll}
L''(\alpha)=0 & \Leftrightarrow \frac{\alpha^2}{2}h''(\mu_{\alpha})=\alpha h'(\mu_{\alpha})-h(\mu_{\alpha})\vspace{2mm}\\
 & \Leftrightarrow h''(\mu_{\alpha})=2L'(\alpha)\vspace{2mm}\\
 & \Leftrightarrow P(-t_{\alpha}log|T'|)=-\frac{\alpha}{2}t'_{\alpha},
\end{array}
$
\end{center}
in which $\mu_\alpha$ is the equilibrium measure for the potential $-t_\alpha\log |T'|$ and all the differentials except $|T'|$  are with respect to $\alpha$. In the case of piecewise linear maps, the condition is expressed explicitly by  Jenkinson, Pollicott and Vytnova as \cite[(17)]{JPV}. Here we tend to deduce similar condition directly from (\ref{eq1}) and (\ref{eq2}), to describe more concavity-convexity  information besides the existence of Lyapunov inflections. For a differentiable function, we say it is \emph{concave at some point} if there exists a small neighbourhood of the point such that the function is concave on the neighbourhood, so as to \emph{convex at some point}. Now we define the concavity-convexity \emph{characteristic function} to be 
\begin{equation}
G(t)=2\log F(t)-\cfrac{(F'(t))^2}{F''(t)F-(F'(t))^2}
\end{equation}
for $t\in[-\infty,\infty]$. The function takes its first appearance in \cite[(17)]{JPV}, it describes the concavity-convexity property of the graph of $L(\alpha)$.
\begin{lem}\label{lem1}
$L(\alpha(t))$ is concave at $\alpha(t)$ if and only if $G(t)<0$, it is convex at $\alpha(t)$ if and only if $G(t)>0$, $\alpha(t)$ is an inflection point  if and only if $G(t)=0$.

\end{lem}
Compare it with \cite[Proposition 2.27]{JPV}.
\begin{proof}
Direct computations show that
\begin{center}
$\cfrac{dL}{dt}=\log F(t)-\cfrac{F''(t)F(t)\log F(t)}{(F'(t))^2},$
\end{center}
\begin{center}
$\cfrac{d\alpha}{dt}=\cfrac{F''(t)F(t)-(F'(t))^2}{F^2(t)},$
\end{center}
so we have\\

$
\begin{array}{ll}
\cfrac{dL}{d\alpha} & =\cfrac{dL}{dt}\cfrac{dt}{d\alpha}=\Big(\log F(t)-\cfrac{F''(t)F(t)\log F(t)}{(F'(t))^2}\Big)\cfrac{F^2(t)}{F''(t)F(t)-(F'(t))^2}\vspace{2mm}\\
 & =-\cfrac{F^2(t)\log F(t)}{(F'(t))^2}.
\end{array}
$\\

The second derivative is \\

$
\begin{array}{lll}
\cfrac{d^2L}{d\alpha^2} & =\cfrac{d(\frac{dL}{d\alpha})}{dt}\cfrac{dt}{d\alpha}=\cfrac{2F(t)(F''(t)F(t)-(F'(t))^2)-(F'(t))^3F(t)}{(F'(t))^2}\cfrac{F^2(t)}{F''(t)F(t)-(F'(t))^2}\vspace{2mm}\\
 & =\cfrac{F^3(t)}{(F'(t))^3}\Big(2\log F(t)-\cfrac{(F'(t))^2}{F''(t)F(t)-(F'(t))^2}\Big)\vspace{2mm}\\
 & =\cfrac{F^3(t)}{(F'(t))^3}G(t).
\end{array}
$\\

Since $\cfrac{F^3(t)}{(F'(t))^3}>0$ for any  $t\in[-\infty,\infty]$, the conclusions follow from \cite[Theorem 17.13, 17.14 ]{HS} or \cite[Theorem 4.4]{Roc}. 
\end{proof}

So we can count the number of Lyapunov inflections by counting the the number of zeros of the map $G(t)$ for piecewise linear maps. Let
\begin{center}
$\bar{G}(t)=2\log F(t)\big(F''(t)F(t)-(F'(t))^2\big)-(F'(t))^2$, 
\end{center}
we have 
\begin{center}
$G(t)=\cfrac{\bar{G}(t)}{F''(t)F(t)-(F'(t))^2}$.
\end{center}
Note that since
\begin{center}
$F''(t)F(t)-(F'(t))^2= \sum_{1\leq i< j\leq n}x_i^t x_j^t(\log x_i-\log x_j)^2>0$,
\end{center}
unless all the branches have the same length of their intervals, the conclusion of Lemma \ref{lem1} holds if $G(t)$ is replaced by $\bar{G}(t)$. A direct application of the lemma shows (see \cite[P539, L3]{IK}) the following result.
\begin{cor}[Iommi-Kiwi]\label{cor1}
$L(\alpha)$ is always concave at the two terminals $\log x_1$ and $\log x_n$.
\end{cor}
\begin{proof}
Simply observe that $\alpha(-\infty)=\log x_1$ and $\alpha(\infty)=\log x_2$, moreover, we have
\begin{center}
$\lim_{t\rightarrow-\infty} G(t)=\lim_{t\rightarrow\infty} G(t)=-\infty$,
\end{center}
so the conclusion follows from Lemma \ref{lem1}.
\end{proof}

\section{Upper bounds on the number of Lyapunov inflections}\label{sec4}

In this section we establish some upper bounds on the the number of Lyapunov inflections for  piecewise linear expanding maps, including Theorem \ref{thm5} and \ref{thm1}. We first consider the situation in words of the branch number $n$, then in words of the essential branch number $r_T$.  We resort to Lemma \ref{lem1}, so we need to estimate the number of zeros of the characteristic function $G(t)$.  It is not quite convenient to estimate the number of zeros of $G(t)$ directly, since  $G(t)=0$ is not a quite normal equation. Alternatively we turn to its differential
\begin{equation}
G'(t)=\cfrac{(F'(t))^2}{F(t)\big(F''(t)F(t)-(F'(t))^2\big)^2}\Big(F^2(t)F'''(t)+2(F'(t))^3-3F(t)F'(t)F''(t)\Big).
\end{equation}
The zeros of $G'(t)$ is the same as zeros of
\begin{equation}
H(t)=F^2(t)F'''(t)+2(F'(t))^3-3F(t)F'(t)F''(t),
\end{equation}
more importantly, $H(t)=0$ belongs to a class of well-known equation-\emph{equation of exponential sums}, see for example \cite{Mor, Tos}. To see this, let
\begin{center}
$
\begin{array}{lll}
Q(i,j,k)=&2\big((\log x_i)^3+(\log x_j)^3+(\log x_k)^3\big)+12\log x_i\log x_j\log x_k-3\big(\log x_i(\log x_j)^2\vspace{2mm}\\

 & +(\log x_i)^2\log x_j+(\log x_i)^2\log x_k+\log x_i(\log x_k)^2+\log x_j(\log x_k)^2\vspace{2mm}\\
 & +(\log x_j)^2\log x_k\big)
\end{array}
$
\end{center}
for any $\{i,j,k\}\subset\{1,\ldots,n\}$ with $i<j<k$ and $n\geq 3$. Then for an $n$-branch piecewise linear expanding map with $n\geq 3$, we can express  $H(t)$ as 
\begin{equation}\label{eq3}
\begin{array}{lll}
H(t) & = & (\sum_{i=1}^n x_i^t)^2(\sum_{i=1}^n x_i^t(\log x_i)^3)+2(\sum_{i=1}^n x_i^t\log x_i)^3\vspace{2mm}\\
     &   & -3(\sum_{i=1}^n x_i^t)(\sum_{i=1}^n x_i^t\log x_i)(\sum_{i=1}^n x_i^t(\log x_i)^2)\vspace{2mm}\\
 & = & (x_1^2 x_2)^t(\log x_2-\log x_1)^3+(x_1 x_2^2)^t(\log x_1-\log x_2)^3+\ldots\vspace{2mm}\\
 &   & +(x_{n-1}^2 x_n)^t(\log x_n-\log x_{n-1})^3+(x_{n-1} x_n^2)^t(\log x_{n-1}-\log x_n)^3\vspace{2mm}\\
 &   &+(x_1 x_2 x_3)^t Q(1,2,3)+\cdots+(x_{n-2} x_{n-1} x_n)^t Q(n-2,n-1,n)\vspace{2mm}\\
 & = & \sum_{1\leq i<j\leq n} \big((x_i^2 x_j)^t(\log x_j-\log x_i)^3+(x_i x_j^2)^t(\log x_i-\log x_j)^3\big)\vspace{2mm}\\
  &   & +\sum_{1\leq i<j<k\leq n}(x_i x_j x_k)^t Q(i,j,k).
 
\end{array}
\end{equation}
If we consider the function in variables of the exponentials, then all its coefficients belong to the $3$-order  field generated by $\{\log x_1, \log x_2, \ldots, \log x_n\}$ over $\mathbb{Z}$. Note that the term Q(i,j,k) can be positive or negative, depending on the exact values of $\log x_i, \log x_j$ and  $\log x_k$.  Considering (\ref{eq3}), we can give a coarse upper bound on the number of Lyapunov inflections for piecewise linear maps-Theorem  \ref{thm5}, by \cite[Theorem 1]{Tos}. Let $C_i^j$ be the number of choices of selecting $j$ ones from $i$ items for $i\geq j$.\\

\noindent{Proof of Theorem  \ref{thm5}:}

\begin{proof}
From the formula (\ref{eq3}), we can see that $H(t)=0$ is an equation of exponential sums with at most
\begin{center}
$2C_n^2+C_n^3=\cfrac{n(n-1)(n+4)}{6}$
\end{center} 
different base numbers. By \cite[Theorem 1]{Tos}, $H(t)$ and so $G'(t)$ has at most 
\begin{center}
$\cfrac{n(n-1)(n+4)}{6}-1$ 
\end{center}
zeros. So $G(t)$ has at most $\cfrac{n(n-1)(n+4)}{6}$ zeros, which implies our conclusion considering Lemma \ref{lem1}. 
\end{proof}
In case of $n=2$, we have $\frac{n(n-1)(n+4)}{6}=2$, which has been shown to be the best upper bound by  Jenkinson, Pollicott and Vytnova  \cite[Theorem 1.3]{JPV}. In case of  $n=3$, we have $\frac{n(n-1)(n+4)}{6}=7$ (in fact it can be decreased to $6$ since the number of Lyapunov inflections for a piecewise linear map is even).  This is still far from our aim on the bound of the number of Lyapunov inflections in Theorem \ref{thm1}, now we sharpen it to the best bound $2$. \\

\noindent Proof of Theorem \ref{thm1}: 
   
\begin{proof}
In the 3-branch case we have 
\begin{center}
$
\begin{array}{lllllllll}
H(t)& = & (x_1^2 x_2)^t(\log x_2-\log x_1)^3-(x_1 x_2^2)^t(\log x_2-\log x_1)^3\vspace{2mm}\\
    & &+(x_1^2 x_3)^t(\log x_3-\log x_1)^3-(x_1 x_3^2)^t(\log x_3-\log x_1)^3\vspace{2mm}\\
    & & +(x_2^2 x_3)^t(\log x_3-\log x_2)^3-(x_2 x_3^2)^t(\log x_3-\log x_2)^3\vspace{2mm}\\
    & & + (x_1 x_2 x_3)^t Q(1,2,3)\vspace{2mm}\\
    & =& (x_1 x_2 x_3)^t \Big((\frac{x_1}{x_3})^t(\log x_2-\log x_1)^3-(\frac{x_2}{x_3})^t(\log x_2-\log x_1)^3  \vspace{2mm}\\
    & & +(\frac{x_1}{x_2})^t(\log x_3-\log x_1)^3-(\frac{x_3}{x_2})^t(\log x_3-\log x_1)^3\\
    & & +(\frac{x_2}{x_1})^t(\log x_3-\log x_2)^3-(\frac{x_3}{x_1})^t(\log x_3-\log x_2)^3+Q(1,2,3)\Big).
\end{array}
$
\end{center}

Now let 
\begin{center}
$
\begin{array}{lll}
\bar{H}(t)= & (\frac{x_1}{x_3})^t(\log x_2-\log x_1)^3-(\frac{x_2}{x_3})^t(\log x_2-\log x_1)^3 +(\frac{x_1}{x_2})^t(\log x_3-\log x_1)^3 \vspace{2mm}\\
    & -(\frac{x_3}{x_2})^t(\log x_3-\log x_1)^3+(\frac{x_2}{x_1})^t(\log x_3-\log x_2)^3-(\frac{x_3}{x_1})^t(\log x_3-\log x_2)^3\vspace{2mm}\\
    & +Q(1,2,3).
\end{array}
$
\end{center}

Obviously $H(t)$ and $\bar{H}(t)$ have the same zeros since $(x_1 x_2 x_3)^t >0$. Differentiate $\bar{H}(t)$, we have
\begin{equation}\label{eq11}
\begin{array}{llllll}
\bar{H}'(t)= & -(\frac{x_1}{x_3})^t(\log x_2-\log x_1)^3(\log x_3-\log x_1)\vspace{2mm}\\
 & +(\frac{x_2}{x_3})^t(\log x_2-\log x_1)^3(\log x_3-\log x_2)\vspace{2mm}\\
 & -(\frac{x_1}{x_2})^t(\log x_3-\log x_1)^3(\log x_2-\log x_1)\vspace{2mm}\\
 & -(\frac{x_3}{x_2})^t(\log x_3-\log x_1)^3(\log x_3-\log x_2) \vspace{2mm}\\
 & +(\frac{x_2}{x_1})^t(\log x_3-\log x_2)^3(\log x_2-\log x_1)\vspace{2mm}\\
 &-(\frac{x_3}{x_1})^t(\log x_3-\log x_2)^3(\log x_3-\log x_1).
\end{array}
\end{equation}

We claim that 
\begin{equation}\label{eq6}
\bar{H}'(t)<0
\end{equation}
for any $t\in(-\infty,\infty)$. This is because either of the two plus terms (also positive terms) can be controlled by a minus term  (also an negative term)  in the expression. If $t\geq 0$, the term
\begin{center}
$(\frac{x_2}{x_3})^t(\log x_2-\log x_1)^3(\log x_3-\log x_2)$
\end{center} 
is controlled by the term
\begin{center}
$-(\frac{x_3}{x_2})^t(\log x_3-\log x_1)^3(\log x_3-\log x_2)$,
\end{center}
and the term
\begin{center}
$(\frac{x_2}{x_1})^t(\log x_3-\log x_2)^3(\log x_2-\log x_1)$
\end{center}
is controlled by the term
\begin{center}
$-(\frac{x_3}{x_1})^t(\log x_3-\log x_2)^3(\log x_3-\log x_1)$.
\end{center}
If $t\leq 0$, the term
\begin{center}
$(\frac{x_2}{x_3})^t(\log x_2-\log x_1)^3(\log x_3-\log x_2)$
\end{center} 
is controlled by the term
\begin{center}
$-(\frac{x_1}{x_3})^t(\log x_2-\log x_1)^3(\log x_3-\log x_1)$,
\end{center}
and the term
\begin{center}
$(\frac{x_2}{x_1})^t(\log x_3-\log x_2)^3(\log x_2-\log x_1)$
\end{center}
is controlled by the term
\begin{center}
$-(\frac{x_1}{x_2})^t(\log x_3-\log x_1)^3(\log x_2-\log x_1)$.
\end{center}

(\ref{eq6}) gives alternative estimations on the number of zeros of  integrations of $\bar{H}'(t)$, which is that, $\bar{H}(t)$ has at most 1 zero, so $H(t)$ and $G'(t)$ have at most 1 zero, which forces that $G(t)$ has at most 2 zeros. Considering Lemma  \ref{lem1} this justifies our theorem.

\end{proof}

The method should be able to be applied or adapted to piecewise linear maps with higher branches, to get some sharper bound on the number of Lyapunov inflections of the corresponding piecewise linear maps, at least in some special cases, for example, see Lemma \ref{lem4}.

Now we consider the upper bound of Lyapunov inflections for  piecewise linear maps in words of the essential branch number $r_T$. We still use $n$ to denote the branch number of $T$, and still assume the non-deceasing of the slopes of branches along with the indexes-(\ref{eq4}). We first point out that in case of $r_T<n$, the upper bound $\frac{n(n-1)(n+4)}{6}$ on the number of Lyapunov inflections in our Theorem  \ref{thm5} can be decreased, as some of the exponential terms have the same base numbers in the expansion of $H(t)$ in (\ref{eq3}). In order to give the exact upper bound, first note that the two terms
\begin{center}
$(x_i^2 x_j)^t(\log x_j-\log x_i)^3+(x_i x_j^2)^t(\log x_i-\log x_j)^3$
\end{center}
disappear in case of $x_i=x_j$. Also not that the term
\begin{center}
$(x_i x_j x_k)^t Q(i,j,k)$
\end{center}
degenerates into
\begin{center}
$(x_i^2 x_k)^t 2(\log x_k-\log x_i)^3$
\end{center}
in case of $x_i=x_j$, it degenerates into
\begin{center}
$-(x_i x_k^2)^t 2(\log x_k-\log x_i)^3$
\end{center}
in case of $x_j=x_k$, it disappears in case of $x_i=x_j=x_k$.

Thus we can bound the number of the Lyapunov inflections for the map $T_N$ in \cite[Theorem 4.2]{JPV} from above by the following result. 

\begin{cor}
For any $N\geq 27$, the piecewise linear map $T_N$ in \cite[Theorem 4.2]{JPV} has at most 
\begin{center}
$2C_{N-5}^2+C_{N-5}^3-S_N$ 
\end{center}
Lyapunov inflections, in which
\begin{center}
$S_N=\left\{\begin{array}{ll}
\sum_{i=1}^{\frac{N-7}{2}} (4i-3)(\frac{N-5}{2}-i)\ \ \ \ \ \ n\ \text{is odd},\vspace{2mm}\\
\sum_{i=1}^{\frac{N-8}{2}} (4i-1)(\frac{N-6}{2}-i)\ \ \ \ \ \  n\ \text{is even}.
\end{array}\right.$
\end{center} 
\end{cor}

\begin{proof}
Again according to \cite[Theorem 1]{Tos}, in order to get the bound, we need to count the number of exponential terms in $H(t)$ with different base numbers. Note that $n=q_N$ now by the symbol in \cite[Theorem 4.2]{JPV}.  First, as for each fixed $6\leq j\leq N$, there are lots of branches with the same slope, so as the explanation above, lots of terms either disappear or degenerate into the same terms with others, so apparently there are at most   
\begin{center}
$2C_{N-5}^2+C_{N-5}^3$
\end{center}
exponential terms with different base numbers in $H(t)$. However, this is still not the best bound, there are still some overlaps of base numbers among these $2C_{N-5}^2+C_{N-5}^3$ exponential terms. For example, by selecting two branches with slope $2^{-2^8}$ and one branch with slope  $2^{-2^{15}}$, we get an exponential term (with coefficients not displayed)
\begin{center}
$2^{-(2^9+2^{15})t}$
\end{center}
in $H(t)$. This is the same as the term we get by selecting one branch with slope $2^{-2^9}$ and two branches with slope  $2^{-2^{14}}$. The number of overlappings of base numbers of exponential terms in this pattern is $S_N$, so there are in fact only
\begin{center}
$2C_{N-5}^2+C_{N-5}^3-S_N$ 
\end{center}
exponential terms with different base numbers in $H(t)$. The remaining proof follows similar scenario with proof of Theorem  \ref{thm5}. According to \cite[Theorem 1]{Tos}, $H(t)$ and $G'(t)$ has at most $2C_{N-5}^2+C_{N-5}^3-S_N-1$ zeros. So $G(t)$ has at most $2C_{N-5}^2+C_{N-5}^3-S_N$ zeros, by Lemma \ref{lem1} this justifies our result.
\end{proof}

Now we turn to piecewise linear expanding maps with low essential branch numbers. In case of $r_T=2$, no matter how many branches $T$ has, we can still bound the number of its Lyapunov inflections by $2$.

\begin{thm}
For a  piecewise linear expanding map $T$ with essential branch number $r_T=2$, the number of its Lyapunov inflections is less than or equal to 2.
\end{thm}  
\begin{proof}
Assume $T$ has $n_1$ branches of slope $x_1$ and $n_2$ branches of slope $x_2$. In this case, the map $H(t)$ degenerates into the following simple form,
\begin{center}
$H(t)=n_1n_2(x_1 x_2)^t(\log x_2-\log x_1)^3(n_1 x_1^t-n_2x_2^t)$.
\end{center}
Obviously $H(t)$ has only one zero, this in turn forces $G(t)$ has at most two zeros in $(-\infty, \infty)$, which justifies the theorem. 
\end{proof}

As stated before, the result is due to Jenkinson, Pollicott and Vytnova \cite[Theorem 1.3]{JPV} in the case of the branch number $r_T=2$. For piecewise linear expanding maps $T$ with $r_T=3$,  we have the following result.

\begin{thm}\label{thm3}
Suppose  $T$ is a  piecewise linear expanding map with essential branch number $r_T=3$, with $n_1$ branches of slope $x_1$, $n_2$ branches of slope $x_2$, and $n_3$ branches of slope $x_3$, then the number of its Lyapunov inflections is less than or equal to 6.
\end{thm}  
\begin{proof}
In this case, the map $H(t)$ degenerates into the following form,
\begin{equation}\label{eq10}
\begin{array}{lllllllllllll}
H(t)& = & C_{n_1}^1C_{n_2}^1\big((x_1^2 x_2)^t(\log x_2-\log x_1)^3-(x_1 x_2^2)^t(\log x_2-\log x_1)^3\big)\vspace{2mm}\\
    & &+C_{n_1}^1C_{n_3}^1\big((x_1^2 x_3)^t(\log x_3-\log x_1)^3-(x_1 x_3^2)^t(\log x_3-\log x_1)^3\big)\vspace{2mm}\\
    & & +C_{n_2}^1C_{n_3}^1\big((x_2^2 x_3)^t(\log x_3-\log x_2)^3-(x_2 x_3^2)^t(\log x_3-\log x_2)^3\big)\vspace{2mm}\\
    & & +2\big(C_{n_1}^2C_{n_2}^1(x_1^2 x_2)^t(\log x_2-\log x_1)^3-C_{n_1}^1C_{n_2}^2(x_1 x_2^2)^t(\log x_2-\log x_1)^3\big)\vspace{2mm}\\
    & &+2\big(C_{n_1}^2C_{n_3}^1(x_1^2 x_3)^t(\log x_3-\log x_1)^3-C_{n_1}^1C_{n_3}^2(x_1 x_3^2)^t(\log x_3-\log x_1)^3\big)\vspace{2mm}\\
    & & +2\big(C_{n_2}^2C_{n_3}^1(x_2^2 x_3)^t(\log x_3-\log x_2)^3-C_{n_2}^1C_{n_3}^2(x_2 x_3^2)^t(\log x_3-\log x_2)^3\big)\vspace{2mm}\\
    & & + C_{n_1}^1C_{n_2}^1C_{n_3}^1(x_1 x_2 x_3)^t Q\vspace{2mm}\\
    &= & n_1^2n_2(x_1^2 x_2)^t(\log x_2-\log x_1)^3-n_1n_2^2(x_1 x_2^2)^t(\log x_2-\log x_1)^3\vspace{2mm}\\
    & &+n_1^2n_3(x_1^2 x_3)^t(\log x_3-\log x_1)^3-n_1n_3^2(x_1 x_3^2)^t(\log x_3-\log x_1)^3\vspace{2mm}\\
    & & +n_2^2n_3(x_2^2 x_3)^t(\log x_3-\log x_2)^3-n_2n_3^2(x_2 x_3^2)^t(\log x_3-\log x_2)^3\vspace{2mm}\\
    & & + n_1n_2n_3(x_1 x_2 x_3)^t Q, \\
\end{array}
\end{equation}
in  which $Q$ is a fixed number. So according to \cite[Theorem 1]{Tos}, $H(t)$ and $G'(t)$ have at most $6$ zeros. So $G(t)$ has at most $7$ zeros, since it can only admit an even number of zeros, by Lemma \ref{lem1} this justifies our theorem. 
\end{proof}

In fact, in general, it is easy to see from the proof above that the upper bound on the number  of Lyapunov inflections still holds if we simply substitute $n$ by $r_T$ in  Theorem  \ref{thm5}.

\begin{cor}
For an $n$-branch piecewise linear expanding map with  essential branch number $r_T$, the  number of its Lyapunov inflections  is less than or equal to 
\begin{center}
$\cfrac{r_T(r_T-1)(r_T+4)}{6}$.
\end{center}

\end{cor}

It is a natural idea to continue to express $H(t)$ in (\ref{eq10}) as we do in (\ref{eq11}), to get better bound in case of $r_T=3$, unfortunately, the addition of coefficients in the terms makes it difficult to control the sign of the corresponding map $\bar{H}'(t)$ as in (\ref{eq6}). We can have some better bounds if the branch numbers $n_1, n_2, n_3$ obey some rules. 

\begin{cor}
In the environment of Theorem \ref{thm3}, if 
\begin{center}
$n_2\leq \min\{n_1, n_3\}$,
\end{center}
then the number of the Lyapunov inflections for the piecewise linear map is less than or equal to 2.
\end{cor}

\begin{proof}
This is because  we can continue to replay the trick in Proof of Theorem \ref{thm1}. Under this assumption on the branch numbers. The positive terms in the corresponding function  $\bar{H}'(t)$ can still be controlled by the negative terms, which forces $H'(t)$ to be strictly negative, so the Proof of Theorem \ref{thm1} revives in this case.
\end{proof}

\section{Root-surgery on the Characteristic function of $3$-branch piecewise linear maps}\label{sec5}

In this section we prove Theorem \ref{thm2} and \ref{thm7}. In fact we will construct a family of  piecewise linear maps with similar interesting property as in Theorem \ref{thm2} and \ref{thm7}. The general idea is as following. We start from some proto-type $3$-branch piecewise linear expanding map with two Lyapunov inflections. We will make a good control on the positions of the two Lyapunov inflections on the $t$-parameter space. Then we do the root-surgery on it-adding a fourth branch with slope large enough into the $3$-branch map. The two Lyapunov inflections of the $3$-branch map will survive after the surgery, moreover, we are guaranteed to create at least two more  Lyapunov inflections in the spectrum of the new $4$-branch map after the surgery. Then  by adding branches with large slopes constantly to do more  surgeries we get  piecewise linear maps with more and more Lyapunov inflections.

Now we start the first step, to construct our proto-type $3$-branch piecewise linear expanding map with its two Lyapunov inflections under desired control. 

\begin{lemma}\label{lem2}
There exists a $3$-branch piecewise linear expanding map $T_{-}$ with two Lyapunov inflections
\begin{center}
$\alpha_{-,1}(t_{-,1})<\alpha_{-,2}(t_{-,2})$,
\end{center}
in which $t_{-,1}, t_{-,2}\in(-\infty, \infty)$ are in the $t$-parameter space, such that 
\begin{center}
$t_{-,1}<t_{-,2}<0$.
\end{center}

\end{lemma}

We  justify the lemma by an explicit  example of $3$-branch piecewise linear expanding maps with these desired properties. 
 
\begin{exa}\label{exa1} 
Let $T_{-}$ be the $3$-branch piecewise linear expanding map with slopes
\begin{center}
$x_1=1.2,\ x_2=19,\ x_3=20$.
\end{center}
\end{exa}

By numerical tests, the piecewise linear map $T_{-}$ has exactly two Lyapunov inflections
\begin{center}
$\alpha_{-,1}(t_{-,1})=1.4038,\ \alpha_{-,2}(t_{-,2})=1.7272$,
\end{center}
in which
\begin{center}
$t_{-,1}=-0.3378,\ t_{-,2}=-0.1706$.
\end{center}
We provide the readers with graphs of its characteristic function $G(t)$ as well as its  Lyapunov spectrum $L(\alpha)$ in Figure \ref{fig3} and \ref{fig4}.
\begin{figure}[h]
\centerline{\includegraphics[scale=0.5]{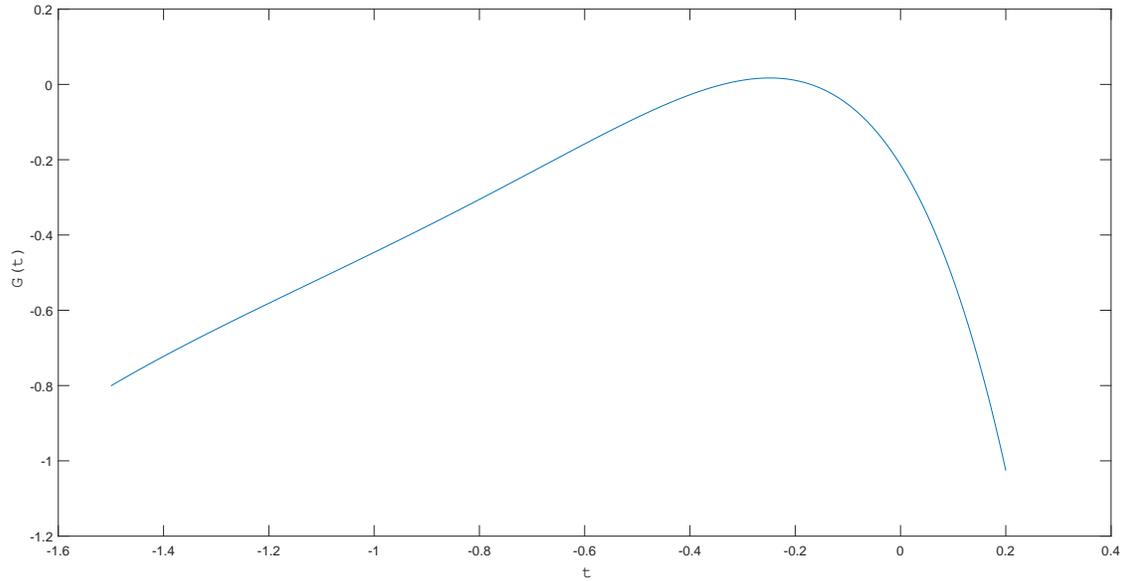}}
\caption{ Graph of $G(t)$ for the $3$-branch map $T_{-}$}
\label{fig3}
\end{figure}

\begin{figure}[h]
\centerline{\includegraphics[scale=0.5]{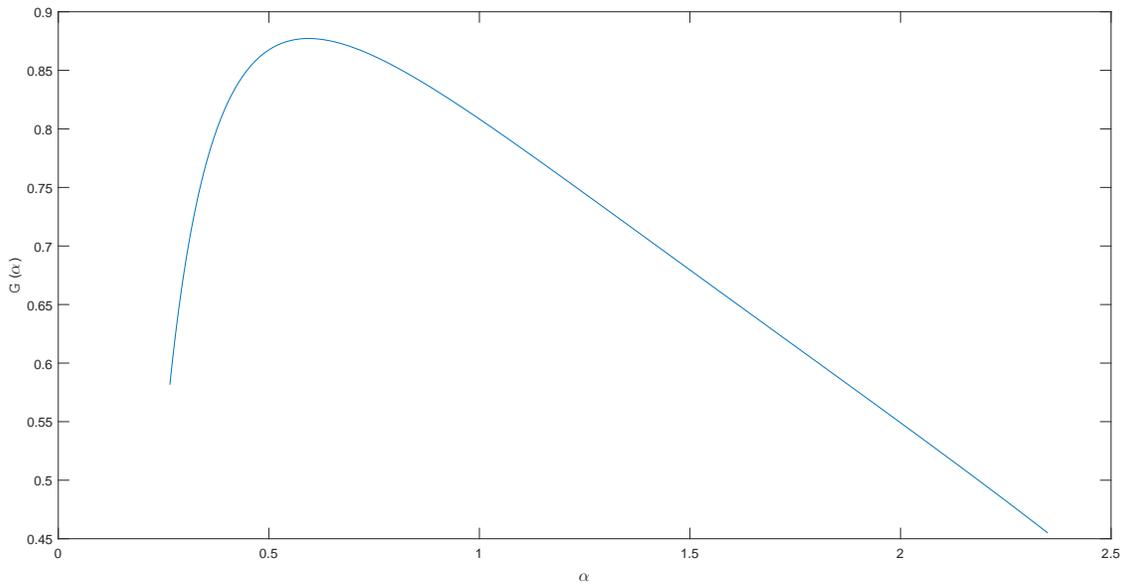}}
\caption{ Graph of $L(\alpha)$ for the $3$-branch map $T_{-}$}
\label{fig4}
\end{figure}

We would like to say some ideas giving birth to the map $T_{-}$ in Example \ref{exa1}. First, due to Corollary \ref{cor2}, there are no  $2$-branch piecewise linear map with  both of its two  Lyapunov inflections parametrized by two negative numbers in the $t$-parameter space. This is because the unique zero of the function $H(t)$ for a $2$-branch piecewise linear map (and hence the derivative of the characteristic function $G(t)$) is $0$, which forces the two Lyapunov inflections to be parametrized by two parameters of different signs, in case of their existence. If we add a third branch of the same slope with the second branch into the $2$-branch system, the corresponding map $H(t)$ degenerates into the form
\begin{center}
$H(t)=2(x_1 x_2)^t(\log x_2-\log x_1)^3(x_1^t-2x_2^t)$.
\end{center}
Its unique zero is at
\begin{center}
$\cfrac{\log 2}{\log x_1-\log x_2}<0$,
\end{center}
so there is opportunity for both of the two Lyapunov inflections to be parametrized by two negative parameters simultaneously. Our numerical tests realize the imagination, and then the proto-type $3$-branch  piecewise linear map $T_{-}$ is got by mild perturbation on the slope of the second branch. 

It is easy to see the following dual conclusion of Lemma \ref{lem2}.

\begin{lemma}\label{lem3}
There exists a $3$-branch piecewise linear expanding map $T_{+}$ with two Lyapunov inflections
\begin{center}
$\alpha_{+,1}(t_{+,1})<\alpha_{+,2}(t_{+,2})$,
\end{center}
in which $t_{+,1}, t_{+,2}\in(-\infty, \infty)$ are in the $t$-parameter space, such that 
\begin{center}
$0<t_{+,1}<t_{+,2}$.
\end{center}

\end{lemma}

We also provide a concrete example-a $3$-branch piecewise linear map satisfying the properties in Lemma \ref{lem3}.

\begin{exa}
Let $T_{+}$ be the $3$-branch piecewise linear expanding map with slopes
\begin{center}
$x_1=3,\ x_2=4,\ x_3=80$.
\end{center}
\end{exa}

The piecewise linear map $T_{+}$ has exactly two Lyapunov inflections
\begin{center}
$\alpha_{+,1}(t_{+,1})=2.4910,\ \alpha_{+,2}(t_{+,2})=3.0781$,
\end{center}
in which
\begin{center}
$t_{+,1}=0.0881,\ t_{+,2}=0.3289$.
\end{center}
The graphs of its characteristic function $G(t)$ and its  Lyapunov spectrum $L(\alpha)$ are illustrated in Figure \ref{fig1} and \ref{fig2}.
\begin{figure}[h]
\centerline{\includegraphics[scale=0.5]{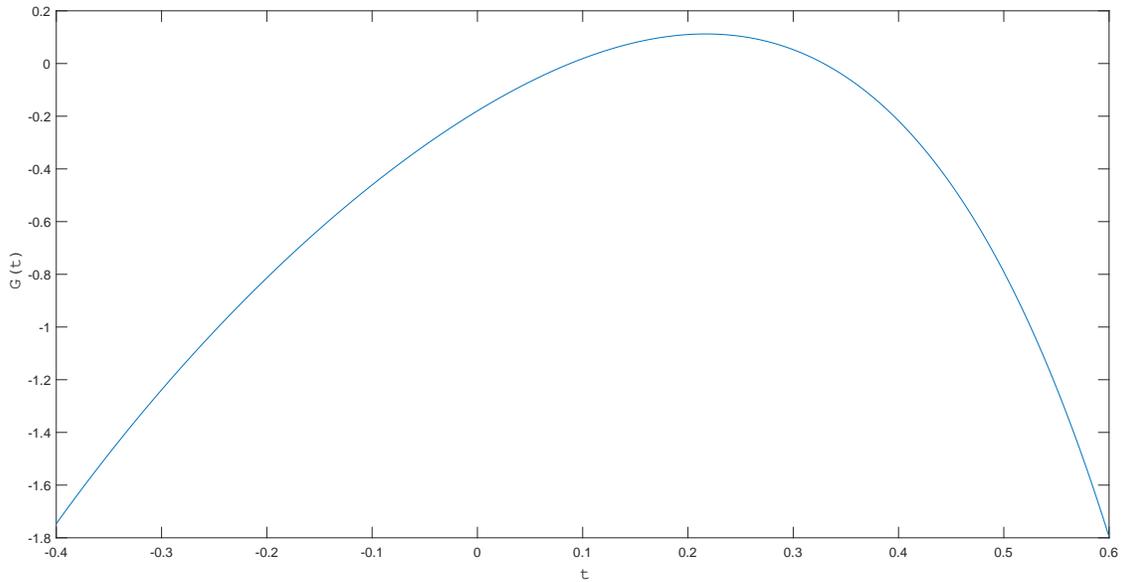}}
\caption{ Graph of $G(t)$ for the $3$-branch map $T_{+}$}
\label{fig1}
\end{figure}

\begin{figure}[h]
\centerline{\includegraphics[scale=0.53]{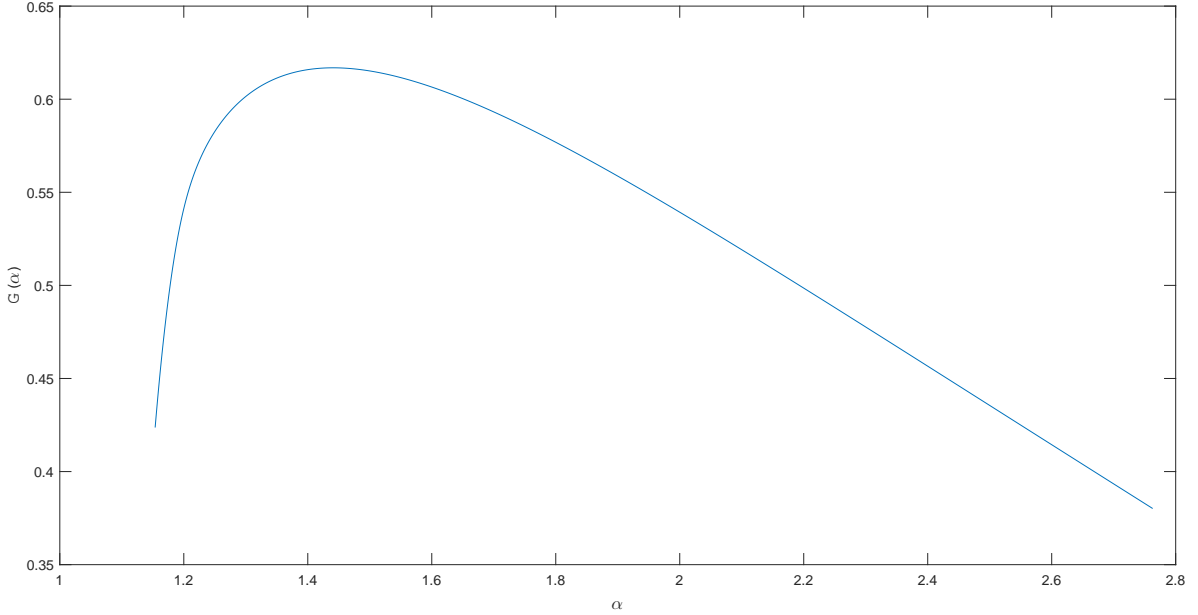}}
\caption{ Graph of $L(\alpha)$ for the $3$-branch map $T_{+}$}
\label{fig2}
\end{figure}

Now we continue the second step, to construct a family of $4$-branch piecewise linear maps with more Lyapunov inflections. This is achieved by adding a fourth branch with large slope into the  $3$-branch piecewise linear map $T_{-}$.

\begin{theorem}\label{thm4}
Let $T_{-,4}$ be a $4$-branch piecewise linear map, with its first three branches of the same slopes $x_1, x_2, x_3$ as the slopes of the three branches of the map $T_{-}$ in Lemma \ref{lem2}. Then there exists a real $K_3$ large enough, such that if the slope $x_4$ of the fourth branch of $T_{-,4}$ satisfies
\begin{center}
$x_4>K_3$,
\end{center}
then  $T_{-,4}$  has more than $4$ Lyapunov inflections.
\end{theorem}

\begin{proof}
First we show that the two Lyapunov inflections of the map $T_{-}$ still exist after adding the fourth branch with large enough slope $x_4$ into the 3-branch map $T_{-}$. To do this, let
\begin{center}
$F_3(t)=x_1^t+x_2^t+x_3^t,$
\end{center}
\begin{center}
$F_4(t)=x_1^t+x_2^t+x_3^t+x_4^t$.
\end{center}
Consider the characteristic functions of the two piecewise linear maps $T_{-}$ and $T_{-,4}$,
\begin{center}
$G_3(t)=2\log F_3(t)-\cfrac{(F'_3(t))^2}{F''_3(t)F_3(t)-(F'_3(t))^2}$,
\end{center} 
\begin{center}
$G_4(t)=2\log F_4(t)-\cfrac{(F'_4(t))^2}{F''_4(t)F_4(t)-(F'_4(t))^2}$.
\end{center} 
For the two  Lyapunov inflections $\alpha_{-,1}(t_{-,1})$ and $\alpha_{-,2}(t_{-,2})$ of the map $T_{-}$, there exist two negative reals $t_a, t_b$ and a small positive real number $\delta>0$ in the $t$-parameter space with
\begin{center}
 $t_a<t_{-,1}<t_b<t_{-,2}< t_{-,2}+\delta<0$,
\end{center}  
such that 
\begin{center}
$G_3(t_a)<0,\ G_3(t_b)>0,\ G_3(t_{-,2}+\delta)<0$.
\end{center}
Let 
\begin{center}
$0<A=\min\{|G_3(t_a)|, G_3(t_b), |G_3(t_{-,2}+\delta)|\}$.
\end{center}
It is easy to see that there exists a real $K_1$ large enough, such that if $x_4>K_1$, then
\begin{center}
$|G_4(t)-G_3(t)|<\frac{A}{2}$
\end{center}
on the domain $(-\infty,t_{-,2}+\delta]$. This means that
\begin{center}
$
\begin{array}{rl}
G_4(t_a)= & G_3(t_a)+G_4(t_a)-G_3(t_a)\leq G_3(t_a)+|G_4(t_a)-G_3(t_a)|<-\frac{A}{2}<0,\vspace{2mm}\\

G_4(t_b)= & G_3(t_b)+G_4(t_b)-G_3(t_b)\geq G_3(t_a)-|G_4(t_a)-G_3(t_a)|>\frac{A}{2}>0,\vspace{2mm}\\

G_4(t_{-,2}+\delta) = & G_3(t_{-,2}+\delta)+G_4(t_{-,2}+\delta)-G_3(t_{-,2}+\delta)\leq G_3(t_{-,2}+\delta)\vspace{2mm}\\
& +|G_4(t_{-,2}+\delta)-G_3(t_{-,2}+\delta)|<-\frac{A}{2}<0.
\end{array}
$
\end{center}
So by the mean value theorem, the  characteristic function $G_4(t)$ has at least 2 zeros on $(-\infty,t_{-,2}+\delta]$, which are in fact evolutions of the two zeros of $G_3(t)$. 

Now we show that there are more Lyapunov inflections created by adding a fourth branch with its slope $x_4$ large enough. To do this, consider the value 
\begin{center}
$G_4(0)=2\log 4-\cfrac{(\sum_{i=1}^4\log x_i)^2}{4\sum_{i=1}^4(\log x_i)^2-(\sum_{i=1}^4\log x_i)^2}$.
\end{center}
Note that as $x_4$ tends to $\infty$ we have
\begin{center}
$\lim_{x_4\rightarrow\infty} G_4(0)=2\log 4-\frac{1}{3}>0$.
\end{center}
So there exists a real $K_2$ large enough, such that if $x_4>K_2$, then
\begin{center}
$G_4(0)>0$.
\end{center}
Since we have 
\begin{center}
$G_4(t_{-,2}+\delta)<0$,
\end{center}
and 
\begin{center}
$\lim_{t\rightarrow\infty} G_4(t)=-\infty$,
\end{center}
by the mean value theorem, there is at least one zero of $G_4(t)$ on either of the two domains $(t_{-,2}+\delta,0)$ and $(0,\infty)$.

At last, let 
\begin{center}
$K_3=\max\{K_1, K_2\}$. 
\end{center}
If we take $x_4>K_3$, then by adding the fourth branch with slope $x_4$ into the piecewise linear map $T_{-}$, it is guaranteed that the characteristic function $G_4(t)$ of $T_{-,4}$ has at least 4 zeros, who give us at least 4  Lyapunov inflections by Lemma \ref{lem1}.

\end{proof}

\begin{rem}
Considering the spectrum $L(\alpha(t))$ of any piecewise linear map $T_{-,4}$ with $x_4>K_3$ in Theorem \ref{thm4} on the $t$-parameter space, we can show that $L(\alpha(t))$ has exactly two inflections on $(-\infty,t_{-,2}+\delta)$, it has exactly one  inflection on $(0,\infty)$, but can not show that it has exactly one  inflection on $(t_{-,2}+\delta,0)$, though we believe it does.   
\end{rem}

As stated above, we show the following result which will be used to prove our Theorem \ref{thm7}.

\begin{lemma}\label{lem4}
There exists a large real number $K_5$, such that if $x_4>K_5$, then all the Lyapunov inflections of the resulted map $T_{-,4}$ in Theorem \ref{thm4} are parametrized by negative parameters, except that the largest one is parametrized by a positive parameter on the  $t$-parameter space.
\end{lemma}

\begin{proof}
We replay the trick in Proof of Theorem \ref{thm1}, but only on the domain $(0,\infty)$ this time. For the map $T_{-,4}$, we have 

\begin{center}
$
\begin{array}{lllllllll}
H(t)& = & (x_1^2 x_2)^t(\log x_2-\log x_1)^3-(x_1 x_2^2)^t(\log x_2-\log x_1)^3\vspace{2mm}\\
    & &+(x_1^2 x_3)^t(\log x_3-\log x_1)^3-(x_1 x_3^2)^t(\log x_3-\log x_1)^3\vspace{2mm}\\
    & &+(x_1^2 x_4)^t(\log x_4-\log x_1)^3-(x_1 x_4^2)^t(\log x_4-\log x_1)^3\vspace{2mm}\\
    & & +(x_2^2 x_3)^t(\log x_3-\log x_2)^3-(x_2 x_3^2)^t(\log x_3-\log x_2)^3\vspace{2mm}\\
    & & +(x_2^2 x_4)^t(\log x_4-\log x_2)^3-(x_2 x_4^2)^t(\log x_4-\log x_2)^3\vspace{2mm}\\
    & & +(x_3^2 x_4)^t(\log x_4-\log x_3)^3-(x_3 x_4^2)^t(\log x_4-\log x_3)^3\vspace{2mm}\\
    & & + (x_1 x_2 x_3)^t Q(1,2,3)+ (x_1 x_2 x_4)^t Q(1,2,4)\vspace{2mm}\\
    & & + (x_1 x_3 x_4)^t Q(1,3,4)+ (x_2 x_3 x_4)^t Q(2,3,4)\vspace{2mm}\\
    & =& (x_1 x_2 x_3 x_4)^t \Big((\frac{x_1}{x_3 x_4})^t(\log x_2-\log x_1)^3-(\frac{x_2}{x_3 x_4})^t(\log x_2-\log x_1)^3  \vspace{2mm}\\
    & & +(\frac{x_1}{x_2 x_4})^t(\log x_3-\log x_1)^3-(\frac{x_3}{x_2 x_4})^t(\log x_3-\log x_1)^3\vspace{2mm}\\
    & & +(\frac{x_1}{x_2 x_3})^t(\log x_4-\log x_1)^3-(\frac{x_4}{x_2 x_3})^t(\log x_4-\log x_1)^3  \vspace{2mm}\\
     & & +(\frac{x_2}{x_1 x_4})^t(\log x_3-\log x_2)^3-(\frac{x_3}{x_1 x_4})^t(\log x_3-\log x_2)^3  \vspace{2mm}\\
     & & +(\frac{x_2}{x_1 x_3})^t(\log x_4-\log x_2)^3-(\frac{x_4}{x_1 x_3})^t(\log x_4-\log x_2)^3  \vspace{2mm}\\
     & & +(\frac{x_3}{x_1 x_2})^t(\log x_4-\log x_3)^3-(\frac{x_4}{x_1 x_2})^t(\log x_4-\log x_3)^3  \vspace{2mm}\\
    & & +\frac{1}{x_4^t}Q(1,2,3)+\frac{1}{x_3^t}Q(1,2,4)+\frac{1}{x_2^t}Q(1,3,4)+\frac{1}{x_1^t}Q(2,3,4)\Big)\vspace{2mm}\\
    & =& (x_1 x_2 x_3 x_4)^t \bar{H}(t).
\end{array}
$
\end{center}
Note that 
\begin{equation}
\begin{array}{llllll}
\bar{H}'(t)= & -(\frac{x_1}{x_3 x_4})^t(\log x_2-\log x_1)^3(\log x_3+\log x_4-\log x_1)\vspace{2mm}\\
             &+(\frac{x_2}{x_3 x_4})^t(\log x_2-\log x_1)^3(\log x_3+\log x_4-\log x_2)  \vspace{2mm}\\
    &  -(\frac{x_1}{x_2 x_4})^t(\log x_3-\log x_1)^3(\log x_2+\log x_4-\log x_1)\vspace{2mm}\\
    &  +(\frac{x_3}{x_2 x_4})^t(\log x_3-\log x_1)^3(\log x_2+\log x_4-\log x_3)\vspace{2mm}\\
    &  -(\frac{x_1}{x_2 x_3})^t(\log x_4-\log x_1)^3(\log x_2+\log x_3-\log x_1)\vspace{2mm}\\
    &  -(\frac{x_4}{x_2 x_3})^t(\log x_4-\log x_1)^3(\log x_4-\log x_2-\log x_3)  \vspace{2mm}\\
    &  -(\frac{x_2}{x_1 x_4})^t(\log x_3-\log x_2)^3(\log x_1+\log x_4-\log x_2)\vspace{2mm}\\
    &  +(\frac{x_3}{x_1 x_4})^t(\log x_3-\log x_2)^3(\log x_1+\log x_4-\log x_3)  \vspace{2mm}\\
    &  -(\frac{x_2}{x_1 x_3})^t(\log x_4-\log x_2)^3(\log x_1+\log x_3-\log x_2)\vspace{2mm}\\
    &  -(\frac{x_4}{x_1 x_3})^t(\log x_4-\log x_2)^3(\log x_4-\log x_1-\log x_3) \vspace{2mm}\\
    &  -(\frac{x_3}{x_1 x_2})^t(\log x_4-\log x_3)^3(\log x_1+\log x_2-\log x_3)\vspace{2mm}\\
    &  -(\frac{x_4}{x_1 x_2})^t(\log x_4-\log x_3)^3(\log x_4-\log x_1-\log x_2)  \vspace{2mm}\\
    &  +\frac{1}{x_4^t}Q(1,2,3)(-\log x_4)\vspace{2mm}\\
    &  -\frac{1}{x_3^t}Q(1,2,4)\log x_3\vspace{2mm}\\
    &  -\frac{1}{x_2^t}Q(1,3,4)\log x_2\vspace{2mm}\\
    &  -\frac{1}{x_1^t}Q(2,3,4)\log x_1\vspace{2mm}\\
\end{array}
\end{equation}

We claim that 
\begin{equation}
\bar{H}'(t)<0
\end{equation}
for any $t\in(0,\infty)$. This is because any of the four plus terms (also positive terms) can be controlled by a minus term  (also an negative term)  in the expression, if $x_4$ is large enough. There exists a large number $K_4$, such that if $x_4>K_4$, then for $t\geq 0$, the term
\begin{center}
$+(\frac{x_3}{x_2 x_4})^t(\log x_3-\log x_1)^3(\log x_2+\log x_4-\log x_3)$
\end{center} 
is controlled by the term
\begin{center}
$-(\frac{x_2}{x_1 x_3})^t(\log x_4-\log x_2)^3(\log x_1+\log x_3-\log x_2)$,
\end{center}
the term
\begin{center}
$+(\frac{x_3}{x_2 x_4})^t(\log x_3-\log x_1)^3(\log x_2+\log x_4-\log x_3)$
\end{center}
is controlled by the term
\begin{center}
$-(\frac{x_4}{x_2 x_3})^t(\log x_4-\log x_1)^3(\log x_4-\log x_2-\log x_3)$,
\end{center}
the term
\begin{center}
$+(\frac{x_3}{x_1 x_4})^t(\log x_3-\log x_2)^3(\log x_1+\log x_4-\log x_3)$
\end{center} 
is controlled by the term
\begin{center}
$-(\frac{x_4}{x_1 x_3})^t(\log x_4-\log x_2)^3(\log x_4-\log x_1-\log x_3)$,
\end{center}
the term (note that $Q(1,2,3)<0$)
\begin{center}
$+\frac{1}{x_4^t}Q(1,2,3)(-\log x_4)$
\end{center}
is controlled by the term
\begin{center}
$-\frac{1}{x_3^t}Q(1,2,4)\log x_3$.
\end{center}

Now let $K_5=\max\{K_3, K_4\}$. If $x_4> K_5$, then (\ref{eq6}) implies $\bar{H}(t)$ has at most 1 zero, so $H(t)$ and $G'(t)$ have at most 1 zero on $(0,\infty)$. Since $G(0)>0$ and $\lim_{t\rightarrow\infty} G(t)=-\infty$, this implies that $G(t)$ has at most 1 zero on $(0,\infty)$. Considering Lemma  \ref{lem1} this justifies our result. 
\end{proof}

From the proof of Lemma \ref{lem4}, one can easily see the following result.

\begin{coro}
For any $n$-branch piecewise linear expanding map $T$, there exists a real number $K_6$ large enough, such that if 
\begin{center}
$x_n>K_6 x_{n-1}$, 
\end{center}
then all the Lyapunov inflections of  $T$ are parametrized by negative parameters, except that the largest one is parametrized by a positive parameter on the  $t$-parameter space.
\end{coro}

Now we give a concrete example of a $4$-branch piecewise linear map to justify our Theorem \ref{thm2}.

\begin{exa}
Let $T_{-,*}$ be the $4$-branch piecewise linear expanding map with slopes
\begin{center}
$x_1=1.2,\ x_2=19,\ x_3=20, x_4=e^{100}$.
\end{center}
\end{exa}

The piecewise linear map $T_{-,*}$ has exactly four Lyapunov inflections
\begin{center}
$\alpha_{-,*,1}(t_{-,*,1})=1.4038,\ \alpha_{-,*,2}(t_{-,*,2})=1.7278$,
\end{center}
\begin{center}
$\alpha_{-,*,3}(t_{-,*,3})=1.8338,\ \alpha_{-,*,4}(t_{-,*,4})=85.7605$,
\end{center}
in which
\begin{center}
$t_{-,*,1}=-0.3378,\ t_{-,*,2}=-0.1703,\ t_{-,*,3}=-0.1147,\ t_{-,*,4}=0.0293$.
\end{center}
The graphs of its characteristic function $G(t)$ and its  Lyapunov spectrum $L(\alpha)$ are illustrated in Figure \ref{fig5} and \ref{fig6}.
\begin{figure}[h]
\centerline{\includegraphics[scale=0.5]{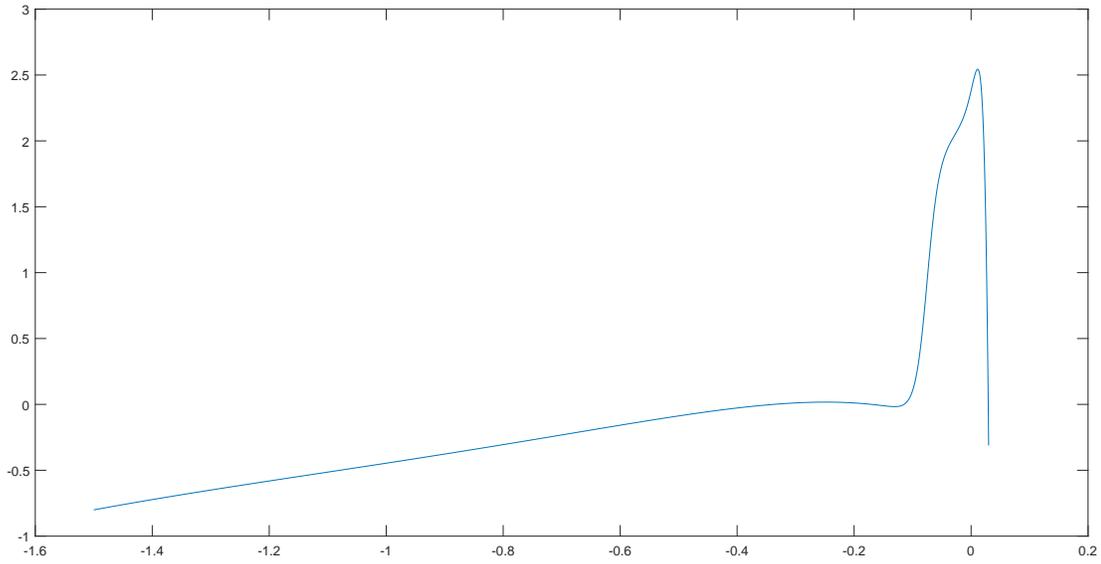}}
\caption{ Graph of $G(t)$ for the $4$-branch map $T_{-,*}$}
\label{fig5}
\end{figure}

\begin{figure}[h]
\centerline{\includegraphics[scale=0.5]{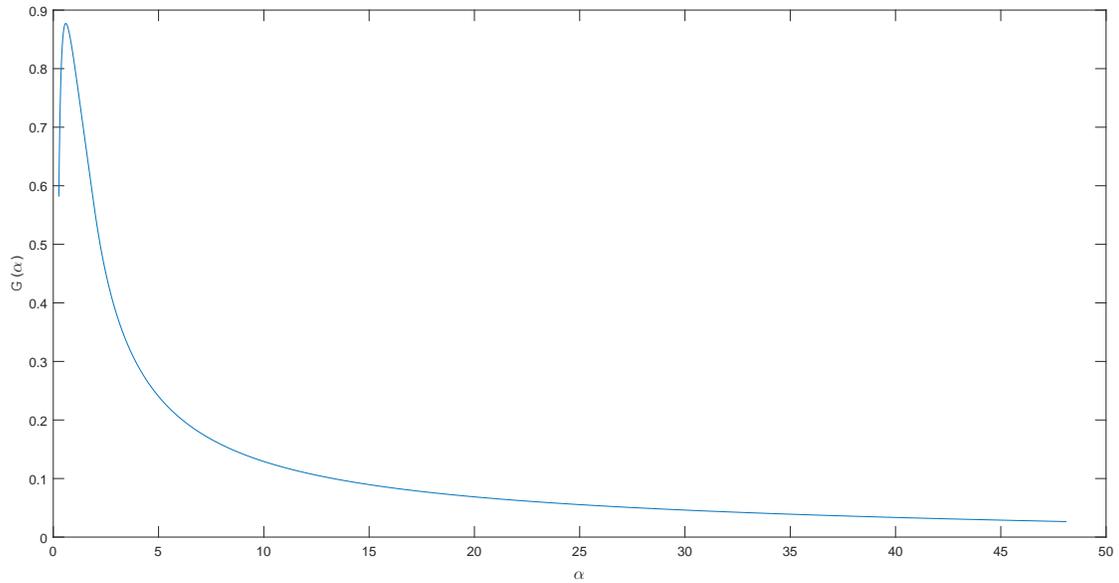}}
\caption{ Graph of $L(\alpha)$ for the $4$-branch map $T_{-,*}$}
\label{fig6}
\end{figure}

Now we are in a position to prove Theorem \ref{thm7}. 

\noindent{Proof of Theorem \ref{thm7}:}

\begin{proof}

We start from a $4$-branch map $T_{-,4}$ with $x_4>K_5$ in Lemma \ref{lem4}. Similar to the Proof of Theorem \ref{thm4} and Proof of Lemma \ref{lem4}, by adding a fifth branch with large enough slope into the $4$-branch map $T_{-,4}$, we get a $5$-branch map  $T_{-,5}$ with $6$ Lyapunov inflections, such that all except the largest one are parametrized by negative parameters on the $t$-parameter space. Now for a fixed $n\geq 6$, continue to do the root-surgery on $T_{-,5}$ for $n-5$ times, we get an $n$-branch piecewise linear map $T_{-,n}$ with $2n-4$ Lyapunov inflections, such that  all but the largest one are parametrized by negative parameters on the $t$-parameter space.   

\end{proof}

By doing the root-surgery on $T_{-,4}$ infinitely many times, we get an $\infty$-branch piecewise linear map with infinitely many Lyapunov inflections, such that the inflections accumulate on $0$ on the $t$-parameter space.

\section{Distribution of the Lyapunov inflections}\label{sec6}

In the case of existence of Lyapunov inflections in piecewise linear systems 
\begin{center}
$T: \cup_{i\in I}X_i\rightarrow [0,1]$, 
\end{center}
how are they distributed in the $\alpha$-parameter space, or by considering parametrization, in the $t$-parameter space? The function $H(t)$ in Section \ref{sec4} still sheds some lights on the possible positions of the Lyapunov inflections, in case of their existence, not only on their numbers. Also note that every distributional result in the $t$-parameter space or the $\alpha$-parameter space can be converted to  a corresponding result on another parameter space, we will intermittently transplant some distributional results from one parameter space into the other one, but not always.  We first show that,

\begin{pro}\label{pro1}
For a piecewise linear expanding map $T$ with at least 3 branches, if 
\begin{center}
$Q(i,j,k)>0$ 
\end{center}
for any $1\leq i<j<k\leq n$, then there is at most one Lyapunov inflection $\alpha(t)$ with parameter  $t\in(-\infty,0]$. Conversely, if 
\begin{center}
$Q(i,j,k)<0$ 
\end{center}
for any $1\leq i<j<k\leq n$, then there is at most one Lyapunov inflection $\alpha(t)$ with parameter  $t\in[0,\infty)$. 
\end{pro}

\begin{proof}
If $G'(t_0)=0$ at the point $t_0$, then $H(t_0)=0$. In this case, transform $H(t_0)=0$ into the following equation,
\begin{equation}
\sum_{1\leq i<j\leq n}(x_i x_j)^{t_0}(\log x_j-\log x_i)^3(x_j^{t_0}-x_i^{t_0})=\sum_{1\leq i<j<k\leq n}(x_i x_j x_k)^{t_0} Q(i,j,k).
\end{equation} 
Now if $Q(i,j,k)>0$ for any $1\leq i<j<k\leq n$, every term in a single bracket is positive, except possibly these terms  
\begin{center}
$(x_j^{t_0}-x_i^{t_0})$ 
\end{center}
for $1\leq i<j\leq n$. Since these terms are all non-negative for $t_0\geq 0$, and all non-positive for  $t_0\leq 0$, then it must be that $t_0> 0$ if $G'(t_0)=0$. So $G(t)$ must be strictly increasing on $(-\infty,0]$, with at most one zero in it, this implies our desired result. 

The converse result can be proved in a similar way.
\end{proof}

\begin{rem}
Note that either of the two conditions in Proposition \ref{pro1} typically can not be satisfied for a piecewise linear map $T$ with essential branch number $r_T<n$. This is because for the combination of branches with slopes $x_i=x_j<x_k$, the term $Q(i,j,k)$ degenerates to 
\begin{center}
$2(\log x_k-\log x_i)^3>0$, 
\end{center}
while for the combination of branches with slopes $x_i<x_j=x_k$, $Q(i,j,k)$ degenerates to 
\begin{center}
$2(\log x_i-\log x_k)^3<0$. 
\end{center}  
However, we can still deduce some results on the location of the least or the largest Lyapunov inflection in this case, for example, see Proposition \ref{pro2}. 
\end{rem}

In words of the slopes of the branches directly, Proposition \ref{pro1}  induces the following result.

\begin{lemma}
There exists a real positive number $C$ large enough, for example, $C=10$, such that if 
\begin{equation}\label{eq8}
\log x_{i+1}>C \log x_i
\end{equation}
for any $1\leq i\leq n-1$, then the corresponding piecewise linear expanding map with branches of  slopes 
$\{x_i\}_{1\leq i\leq n}$ has at most one Lyapunov inflection $\alpha(t)$ with parameter  $t\in(-\infty,0]$. 
\end{lemma}

\begin{proof}
Write $Q(i,j,k)$ as
\begin{equation}
\begin{array}{ll}\label{eq9}
Q(i,j,k)= & 4\big((\log x_i)^3+(\log x_j)^3+(\log x_k)^3\big)+12\log x_i\log x_j\log x_k\vspace{2mm}\\
          & -\big((\log x_i+\log x_j)^3+(\log x_i+\log x_k)^3+(\log x_j+\log x_k)^3\big).         
\end{array}
\end{equation}
For any $1\leq i<j<k\leq n$, if (\ref{eq8}) holds for $C=10$, considering (\ref{eq9}) we have
\begin{center}
$Q(i,j,k)>4(\log x_k)^3-(0.2\log x_k)^3-2(1.1\log x_k)^3>0$.
\end{center}
Then the conclusion follows from Proposition \ref{pro1} instantly.

\end{proof}

We can say more about the distribution of the Lyapunov inflections for  a  piecewise linear expanding map with  low essential branch number, in case of their existence. For example, in case of $r_T=2$, we have the following result.
\begin{pro}\label{pro2}
For a piecewise linear expanding map $T$ with $n_1$ branches of slope $x_1$ and $n_2$ branches of slope $x_2$, let 
\begin{center}
$t_*=(\log\frac{n_2}{n_1})/(\log\frac{x_1}{x_2})$.
\end{center}
If it has two Lyapunov inflections 
\begin{center}
$\alpha_1(t_1)< \alpha_2(t_2)$, 
\end{center}
with $t_1<t_2$ on the $t$-parameter space, then it is necessary that 
\begin{center}
$\alpha_1(t_1)<\alpha(t_*)< \alpha_2(t_2)$.
\end{center}
On the $t$-parameter space, this means $t_1<t_*<t_2$.
\end{pro}
\begin{proof}

In this case, the map $H(t)$ is simplified as
\begin{center}
$H(t)=n_1n_2(x_1 x_2)^t(\log x_2-\log x_1)^3(n_1 x_1^t-n_2x_2^t)$.
\end{center}
It admits a unique zero at $t_*=(\log\frac{n_2}{n_1})/(\log\frac{x_1}{x_2})$. As the zero of $G'(t)$, it forces $t_*$ to be between the two inflections on the $t$-parameter space, in case of their existence.

\end{proof}

In the case of $2$-branch piecewise linear expanding maps, this implies the following result.

\begin{coro}\label{cor2}
For a $2$-branch piecewise linear expanding map  $T$ with distinct slopes $x_1$ and $x_2$,  if it has two Lyapunov inflections $\alpha_1< \alpha_2$, then it is necessary that 
\begin{center}
$\alpha_1<\cfrac{\log x_1+\log x_2}{2}< \alpha_2$.
\end{center} 
\end{coro}
\begin{proof}
In this case we have $t_*=0$, then the result is obvious by Proposition \ref{pro2}, since now
\begin{center}
$\alpha(0)=\cfrac{\log x_1+\log x_2}{2}$.
\end{center}
  
\end{proof}

One can also deduce Corollary \ref{cor2} from the proof of \cite[Theorem 1.3]{JPV} in Jenkinson-Pollicott-Vytnova's work.

Now we consider the distribution of the Lyapunov inflections relative to the positions of the  Lyapunov milestones. For an $n$-branch piecewise linear expanding map $T$ with milestones 
\begin{center}
$\{\log x_1, \log x_2, \cdots, \log x_n\}$,
\end{center}
let $t_1, t_2, \cdots, t_n$ be the corresponding point on the $t$-parameter space with
\begin{center}
$\{\log x_1=\alpha(t_1), \log x_2=\alpha(t_2), \cdots, \log x_n=\alpha(t_n)\}$.
\end{center}
It is known that the least Lyapunov milestone $\log x_1$ and the largest milestone  $\log x_n$ can not be the Lyapunov inflections, as the spectrum $L(\alpha)$ is always concave at the two terminals according to Corollary \ref{cor1}. So it is a natural question to ask, how is the thing at the other milestones $\cup_{2\leq i\leq n-1}\{\log x_i\}$ ? Can they be the Lyapunov inflections or is $L(\alpha)$  always concave at these milestones? We have the following result to these questions.
\begin{theorem}\label{thm6}
For an $n$-branch piecewise linear expanding map, the intermediate milestones 
\begin{center}
$\cup_{2\leq i\leq n-1}\{\log x_i\}$ 
\end{center}
are possible to be the Lyapunov inflections, and $L(\alpha)$ is possible to be strictly convex at the intermediate milestones.
\end{theorem}

So it turns out that the intermediate milestones, not like the two terminal ones, have enough flexibility of concavity-convexity property on the graph of the spectrum $L(\alpha)$. However, I still believe a stronger result similar to  \cite[Theorem A]{IK} which describes the exact bifurcations  from the concavity to the increasing of Lyapunov inflections is possible. In the case of an $n$-branch piecewise linear map $T$, if the family has at most $P_n$ Lyapunov inflections, the combinations of  concavity-inflection is  of at most $\frac{P_n}{2}$-type, so there should be at most  $\frac{P_n}{2}$ types of bifurcations in the parameter space of the slopes of branches.

We will show  Theorem \ref{thm6} by a concrete example of piecewise linear maps.

\begin{theorem}
Let $x_1=1.2, x_3=200$. Among $3$-branch piecewise linear maps, there exists  one number 
\begin{center}
$x_*\in[29.542,29.543]$, 
\end{center}
with slopes of the three branches of the piecewise linear map $T_*$ being $x_1, x_*, x_3$,
such that $T_*$ has 2 Lyapunov inflections, with one of them being exactly the milestone $\log x_*$.
\end{theorem} 

\begin{proof}
Recall the symbol $t_2$ is defined by $\alpha(t_2)=\log x_2$ for any $n$-branch piecewise linear map. We will use the same symbol $t_2$ for different piecewise linear maps, as one can easily understand which map the symbol is associated to from the context. By our numerical test, we have
\begin{center}
$G(t_2)>0$
\end{center}
for the 3-branch map with slopes $x_1=1.2, x_2=29.542, x_3=200$. We have
\begin{center}
$G(t_2)<0$
\end{center}
for the 3-branch map with slopes $x_1=1.2, x_2=29.543, x_3=200$.
So the result follows from the mean value theorem.  

\end{proof}

By our numerical test one can take the value of $x_*$ approximately as 
\begin{center}
$x_*\approx 29.54276$. 
\end{center}

The two Lyapunov inflections are approximately 
\begin{center}
$\alpha_1=\alpha(-0.4218)=1.2159$, \vspace{2mm}\\
$\alpha_2=\alpha(0.1008)=3.3858=\log 29.54276$,
\end{center}
with $\alpha_2$ being an Lyapunov inflection and milestone simultaneously. The graph of $G(t)$ and $L(\alpha)$ for the 3-branch map $T_*$ with slopes $x_1=1.2, x_2=29.54276, x_3=200$ is presented in Figure \ref{fig7} and \ref{fig8}.

\begin{figure}[h]
\centering
\includegraphics[scale=0.75]{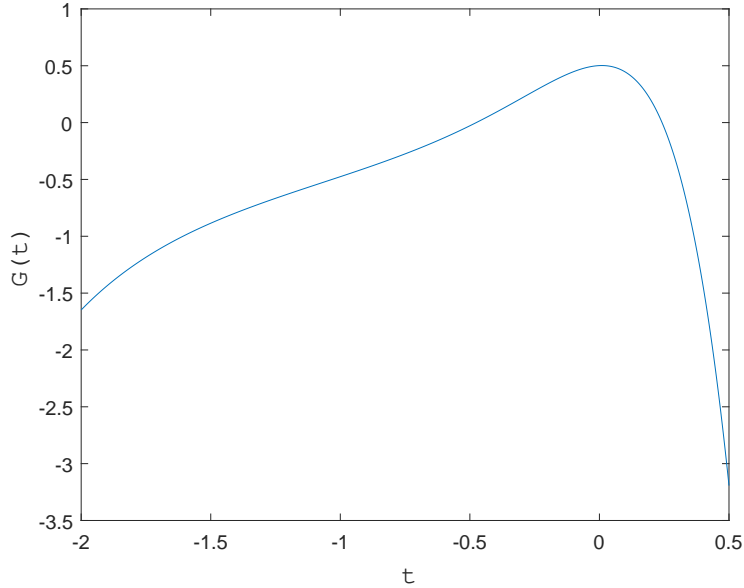}
\caption{ Graph of $G(t)$ for the 3-branch map $T_*$}
\label{fig7}
\end{figure}

\begin{figure}[h]
\centering
\includegraphics[scale=0.75]{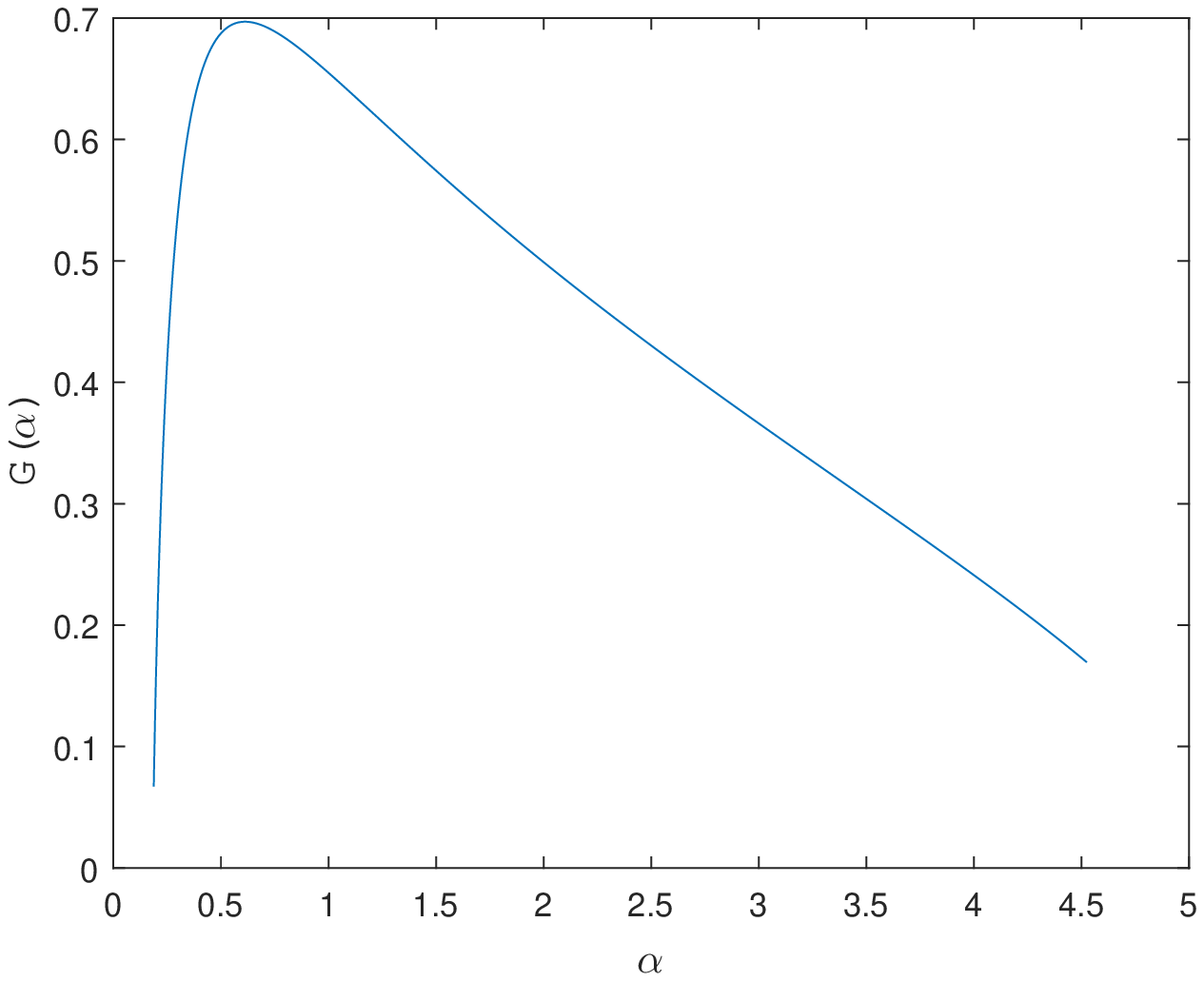}
\caption{ Graph of $L(\alpha)$ for the 3-branch map $T_*$}
\label{fig8}
\end{figure}

\end{document}